\documentclass[12pt]{article}
\usepackage{amsmath,amssymb,amsthm}
\RequirePackage[dvips]{graphicx}
\usepackage{xcolor}
\usepackage{tikz}
\usepackage{cite}

\theoremstyle{plain}
\newtheorem{theorem}{Theorem}[section]

\newtheorem{lemma}[theorem]{Lemma}
\newtheorem{corollary}[theorem]{Corollary}

\theoremstyle{definition}
\newtheorem{remark}[theorem]{Remark}

\begin{document}

  \begin{center}
 {\Large \bf Bounds on the Inverse symmetric division deg index\\[2mm] and the relation with other topological indices of graphs}\\

 \vspace{6mm}

 \bf {Kinkar Chandra Das$^a$, B. R. Rakshith$^{b,}\footnote{Corresponding author}$, Wojciech Macek$^c$}

 \vspace{5mm}

{\it $^a$Department of Mathematics, Sungkyunkwan University, \\
 Suwon 16419, Republic of Korea\/} \\
{\tt E-mail: kinkardas2003@gmail.com}\\[2mm]

{\it $^b$Department of Mathematics, Manipal Institute of Technology, \\
Manipal Academy of Higher Education, Manipal 576 104, India.\/} \\
		{\rm E-mail:} {\tt ranmsc08@yahoo.co.in}\\[2mm]

{\it $^c$Faculty of Mechanical Engineering and Ship Technology, Gda\'nsk University of Technology, \\
 80-233 Gda\'nsk, Poland\/} \\
{\tt E-mail: wojciech.macek@pg.edu.pl}

 \hspace*{4mm}

 \end{center}

 \baselineskip=0.23in

\begin{abstract}
Let $G=(V,E)$ be a simple graph. The concept of Inverse symmetric division deg index $(ISDD)$ was introduced in the chemical graph theory very recently. In spite of this, a few papers have already appeared with this index in the literature. Ghorbani et al. proposed Inverse symmetric division deg index and is defined as
$$ISDD(G)=\sum\limits_{v_iv_j\in E(G)}\,\displaystyle{\frac{d_id_j}{d^2_i+d^2_j}},$$
where $d_i$ is the degree of the vertex $v_i$ in $G$. In this paper, we obtain some lower and upper bounds on the inverse symmetric division deg index $(ISDD)$ of graphs in terms of various graph parameters, with identifying extremal graphs. Moreover, we present two relations between the Inverse symmetric division deg index and the various topological indices of graphs. Finally, we give concluding remarks with future work.

\end{abstract}

\vskip0.2cm
\noindent {\bf Declaration of Competing Interest.} The authors declare no conflict of interest.

\baselineskip=0.30in

\section{Introduction}

Molecular descriptors are essential in mathematical chemistry, especially in QSPR/QSAR investigations. The so-called topological descriptors have a specific place among them. These days, there are a number of topological indices with chemical applications. The structural characteristics of the graphs that were utilized to calculate them can be used to classify them. Numerous indices and their relationship to certain chemical compounds' physical properties have been researched in the literature. Some of the topological indices \cite{E11,E12,E13,WI1,HO1,R,DA0,DA3,GT,DZT,E,CD1,BHD1,CW1,DA5,DA4,FGV,20,21,22,25} are very popular not only in mathematical chemistry but also in graph theory.

\vspace*{2mm}

Let $G=(V,E)$ be a simple graph having $n$ vertices and $m$ edges, where $V(G)=\{v_1,v_2,\ldots,v_n\}$ with $|V(G)|=n$, and $|E(G)|=m$. The degree of the vertex $v_i$, denoted by $d_{i}$, is the number of edges incident to $v_i$. The neighbor set of the vertex $v_i$ is denoted by $N_G(v_i)$, that is, $|N_G(v_i)|=d_i$. We denoted by maximum degree $\Delta(G)=\max_{1\leq i\leq n}\,d_i$, and minimum degree $\delta(G)=\min_{1\leq i\leq n}\,d_i$. We write $v_iv_j\in E(G)$ if vertices $v_i$ and $v_j$ are adjacent in $G$. The book \cite{BM} contains the graph theoretic definitions.

\vspace*{2mm}

Two oldest degree based topological indices are the first Zagreb index $M_1$ and the second Zagreb index $M_2$. The first Zagreb index and the second Zagreb index of a graph $G$ are defined as follows:
$$M_1(G)=\sum\limits_{v_i\in V(G)}\,d^2_i=\sum\limits_{v_iv_j\in E(G)}\,(d_i+d_j)~~\mbox{ and }~~M_2(G)=\sum\limits_{v_iv_j\in E(G)}\,d_i\,d_j.$$
Mathematical properties on the Zagreb indices are reported in \cite{CA,GT,M1,WY,xu2011}, and the references therein.

The forgotten topological index of a graph $G$ is defined as follows:
$$F(G)=\sum\limits_{v_i\in V(G)}\,d^3_i=\sum\limits_{v_iv_j\in E(G)}\,(d^2_i+d^2_j).$$
Some mathematical results on the forgotten topological index is reported in \cite{CC,EMR,FG,LM1}.

The geometric-arithmetic index of a graph $G$ is denoted by $GA(G)$ and is defined as follows:
$$GA(G)=\sum\limits_{v_iv_j\in E(G)}\,\displaystyle{\frac{2\,\sqrt{d_id_j}}{d_i+d_j}}.$$
For its basic properties, including various lower and upper bounds, see \cite{DA1,26,yua10}.

\vspace*{2mm}

Vuki\v{c}evi\'c and Gasperov \cite{VG} looked into a new class of topological indices called the ``discrete Adriatic indices", which consists of $148$ indices, in order to improve QSPR/QSAR studies. One of these few indices is the symmetric division deg $(SDD)$ index, which is defined for a graph $G$ as
 $$SDD(G)=\sum\limits_{v_iv_j\in E(G)}\,\displaystyle{\frac{d^2_i+d^2_j}{d_id_j}}.$$
Furtula et al. conducted a thorough multidimensional analysis of the SDD index in \cite{FDG}, and it was discovered to be a feasible and useful topological index. This index outperformed a number of other well-known topological indices. Since then, a large number of papers specifically pertaining to SDD have been published; for instance, see \cite{AEM,DA4,23}.

\vspace*{2mm}

Ghorbani et al. \cite{M1} introduced the inverse SDD index, ISDD, which stands for Inverse symmetric division deg index. The $ISDD$ index of a graph $G$ is defned as follows:
$$ISDD(G)=\sum\limits_{v_iv_j\in E(G)}\,\displaystyle{\frac{d_id_j}{d^2_i+d^2_j}}.$$
The maximum and minimum trees of fixed order with respect to ISSD index have been characterized completely in \cite{M1}. Albalahi and Ali \cite{AA1} addressed the problem of finding the graphs having the largest and smallest ISDD index from the set of all connected unicyclic graphs having the specified order. In \cite{RSD1}, the authors presented some upper and lower bounds including extremal graphs on $ISDD$ of several class of graphs. A bipartite graph $G$ with a bipartition $U$ and $W$ such that every vertex $v_i$ in $U$ has the same degree $r$, and every vertex $v_j$ in $W$ has the same degree $s$, then $G$ will be called a $(r,s)$-semiregular bipartite graph.

\vspace*{2mm}

This paper is structured as follows. In Section 2, we gave some upper and lower bounds on $ISDD$ index of graphs with various graph parameters, and characterize the corresponding extremal graphs. In Section 3, we present some relations between $ISDD$ index with some popular topological indices of graphs. In Section 4, we give concluding remarks with future work.

\section{On $ISDD$ index of graphs}

In this section we give some lower and upper bounds on $ISDD$ of graphs in terms of different graph parameters, and characterize the corresponding extremal graphs.
The following result is obtained from the proof of Theorem 2.3 in \cite{DA1}.
\begin{lemma} {\rm \cite{DA1}} \label{p1a2} For any edge $v_iv_j\in E(G)\,(d_i\geq d_j)$,
   $$\frac{d_i\,d_j}{d^2_i+d^2_j}\geq \frac{\Delta\,\delta}{\Delta^2+\delta^2},$$
where $\Delta$ is the maximum degree and $\delta$ is the minimum degree in $G$. Moreover, the equality holds if and only if $(d_i,\,d_j)=(\Delta,\,\delta)$.
\end{lemma}
In the above, we obtained the minimum value of $\displaystyle{\frac{d_i\,d_j}{d^2_i+d^2_j}}$ for any edge $v_iv_j\in E(G)$. We now focus on the second minimum value of $\displaystyle{\frac{d_i\,d_j}{d^2_i+d^2_j}}$ in the following result.
\begin{lemma} \label{1a2} For any edge $v_iv_j\in E(G)$ with $(d_i,\,d_j)\neq (\Delta,\,\delta)$,
   $$\frac{d_i\,d_j}{d^2_i+d^2_j}\geq \frac{(\Delta-1)\,\delta}{(\Delta-1)^2+\delta^2},$$
where $\Delta$ is the maximum degree and $\delta$ is the minimum degree in $G$. Moreover, the equality holds if and only if $(d_i,\,d_j)=(\Delta-1,\,\delta)$.
\end{lemma}

\begin{proof} Let $v_iv_j$ be any edge in $G$ with $d_i\geq d_j$. Since $(d_i,\,d_j)\neq (\Delta,\,\delta)$, we consider the following cases:

\vspace*{3mm}

\noindent
${\bf Case\,1.}$ $d_i\neq \Delta$. In this case $\delta\leq d_j\leq d_i\leq \Delta-1$. One can easily see that
$$\displaystyle{\frac{d_i}{d_j}}\leq \displaystyle{\frac{\Delta-1}{\delta}},~\mbox{ that is, }~\sqrt{\displaystyle{\frac{d_i}{d_j}}}\leq \sqrt{\displaystyle{\frac{\Delta-1}{\delta}}}~\mbox{ and }\sqrt{\displaystyle{\frac{d_j}{d_i}}}\geq \sqrt{\displaystyle{\frac{\delta}{\Delta-1}}},$$
which implies
$$\sqrt{\frac{d_i}{d_j}}-\sqrt{\frac{d_j}{d_i}}\leq \sqrt{\frac{\Delta-1}{\delta}}-\sqrt{\frac{\delta}{\Delta-1}},$$
that is,
$$\left(\sqrt{\frac{d_i}{d_j}}-\sqrt{\frac{d_j}{d_i}}\right)^2\leq \left(\sqrt{\frac{\Delta-1}{\delta}}-\sqrt{\frac{\delta}{\Delta-1}}\right)^2,$$
that is,
$$\frac{d_i}{d_j}+\frac{d_j}{d_i}\leq \frac{\Delta-1}{\delta}+\frac{\delta}{\Delta-1},$$
that is,
\begin{equation}
\frac{d_i\,d_j}{d^2_i+d^2_j}\geq \frac{(\Delta-1)\,\delta}{(\Delta-1)^2+\delta^2}.\label{1kk1}
\end{equation}
Moreover, the equality holds in (\ref{1kk1}) if and only if $(d_i,\,d_j)=(\Delta-1,\,\delta)$.

\vspace*{3mm}

\noindent
${\bf Case\,2.}$ $d_j\neq \delta$. In this case $\delta+1\leq d_j\leq d_i\leq \Delta$. Similarly, as before, one can easily obtain that
$$\displaystyle{\frac{d_i}{d_j}}\leq \displaystyle{\frac{\Delta}{\delta+1}},~\mbox{ that is, }~\sqrt{\displaystyle{\frac{d_i}{d_j}}}\leq \sqrt{\displaystyle{\frac{\Delta}{\delta+1}}}~\mbox{ and }\sqrt{\displaystyle{\frac{d_j}{d_i}}}\geq \sqrt{\displaystyle{\frac{\delta+1}{\Delta}}},$$
which implies
\begin{equation}
\frac{d_i\,d_j}{d^2_i+d^2_j}\geq \frac{\Delta\,(\delta+1)}{\Delta^2+(\delta+1)^2}.\label{1kk2}
\end{equation}
Moreover, the equality holds in (\ref{1kk2}) if and only if $(d_i,\,d_j)=(\Delta,\,\delta+1)$.

\vspace*{3mm}

\noindent
${\bf Case\,3.}$ $d_i\neq \Delta$ and $d_j\neq \delta$. In this case $\delta+1\leq d_j\leq d_i\leq \Delta-1$. Similarly, as {\bf Case\,1}, one can easily obtain that
\begin{equation}
\frac{d_i\,d_j}{d^2_i+d^2_j}\geq \frac{(\Delta-1)\,(\delta+1)}{(\Delta-1)^2+(\delta+1)^2}.\label{1kk3}
\end{equation}
Moreover, the equality holds in (\ref{1kk3}) if and only if $(d_i,\,d_j)=(\Delta-1,\,\delta+1)$.

\vspace*{3mm}

\noindent
${\bf Claim\,1.}$
$$\frac{(\Delta-1)\,\delta}{(\Delta-1)^2+\delta^2}\leq \frac{\Delta\,(\delta+1)}{\Delta^2+(\delta+1)^2}\leq  \frac{(\Delta-1)\,(\delta+1)}{(\Delta-1)^2+(\delta+1)^2}.$$

\vspace*{3mm}

\noindent
{\bf Proof of Claim 1.} First we prove the left inequality, that is, we have to prove that
        $$\frac{(\Delta-1)^2+\delta^2}{(\Delta-1)\,\delta}\geq \frac{\Delta^2+(\delta+1)^2}{\Delta\,(\delta+1)},$$
that is,
$$\frac{\Delta-1}{\delta}+\frac{\delta}{\Delta-1}\geq \frac{\Delta}{\delta+1}+\frac{\delta+1}{\Delta},$$
that is,
\begin{equation}
\frac{\Delta-\delta-1}{\delta\,(\delta+1)}\geq \frac{\Delta-\delta-1}{\Delta\,(\Delta-1)}.\label{1kk4}
\end{equation}
Since $(d_i,\,d_j)\neq (\Delta,\,\delta)$, we have $\Delta\neq \delta$, that is, $\Delta\geq \delta+1$. If $\Delta=\delta+1$, then (\ref{1kk4}) holds. Otherwise, $\Delta\geq \delta+2$. Then $\Delta\,(\Delta-1)>\delta\,(\delta+1)$, and hence (\ref{1kk4}) strictly holds.

\vspace*{3mm}

Next we prove the right inequality. Since $d_i\neq \Delta$ and $d_j\neq \delta$, we obtain $\delta+1\leq d_j\leq d_i\leq \Delta-1$. We have to prove that
 $$\frac{\Delta\,(\delta+1)}{\Delta^2+(\delta+1)^2}\leq  \frac{(\Delta-1)\,(\delta+1)}{(\Delta-1)^2+(\delta+1)^2},~\mbox{ that is, }~\frac{\Delta-1}{\delta+1}+\frac{\delta+1}{\Delta-1}\leq \frac{\Delta}{\delta+1}+\frac{\delta+1}{\Delta},$$
that is,
  $$\Delta\,(\Delta-1)\geq (\delta+1)^2,$$
which is always true as $\Delta\geq \delta+2$. This proves the {\bf Claim 1}.

\vspace*{3mm}

From (\ref{1kk1})--(\ref{1kk3}) with {\bf Claim 1}, we obtain
$$\frac{d_i\,d_j}{d^2_i+d^2_j}\geq \frac{(\Delta-1)\,\delta}{(\Delta-1)^2+\delta^2}$$
with equality if and only if $(d_i,\,d_j)=(\Delta-1,\,\delta)$.
\end{proof}

\begin{corollary} \label{1a3} Let $T$ be a tree of order $n$ and also let $v_iv_j\in E(T)$ be any edge. If $(d_i,\,d_j)\neq (n-1,\,1)$, then
   $$\frac{d_i\,d_j}{d^2_i+d^2_j}\geq \frac{n-2}{(n-2)^2+1}>\frac{n-1}{(n-1)^2+1}.$$
\end{corollary}

\begin{proof} Let $v_iv_j\in E(T)$ be any edge with $d_i\geq d_j$. Since $(d_i,\,d_j)\neq (n-1,\,1)$, we obtain, $n>3$. One can easily check that
  $$\frac{n-2}{(n-2)^2+1}>\frac{(n-1)}{(n-1)^2+1}.$$

Since $T$ is a tree, we have $\delta=1$. By Lemma \ref{1a2}, for $(d_i,\,d_j)\neq (\Delta,\,1)$, we obtain
             $$\frac{d_i\,d_j}{d^2_i+d^2_j}\geq \frac{\Delta-1}{(\Delta-1)^2+1}.$$

We have to prove that
  $$\frac{\Delta-1}{(\Delta-1)^2+1}\geq \frac{n-2}{(n-2)^2+1},$$
that is,
  $$n-2+\frac{1}{n-2}\geq \Delta-1+\frac{1}{\Delta-1},$$
which is always true as $\Delta\leq n-1$. Hence
    $$\frac{d_i\,d_j}{d^2_i+d^2_j}\geq \frac{n-2}{(n-2)^2+1}>\frac{(n-1)}{(n-1)^2+1}$$
This completes the proof of the result.
\end{proof}

\begin{figure}[ht!]
\begin{center}
\begin{tikzpicture}[scale=0.5,style=thick]
\def\vr{7pt}
\def\br{0.6pt}

\draw (0,0)-- (1,2);
\draw (1,0)-- (1,2);
\draw (2,0)-- (1,2);
\draw (3,0)-- (4,2);
\draw (4,0)-- (4,2);
\draw (5,0)-- (4,2);
\draw (6,0)-- (7,2);
\draw (7,0)-- (7,2);
\draw (8,0)-- (7,2);
\draw (9,0)-- (10,2);
\draw (10,0)-- (10,2);
\draw (11,0)-- (10,2);

\draw (0,0)-- (2.5,-2);
\draw (1,0)-- (2.5,-2);
\draw (2,0)-- (2.5,-2);
\draw (3,0)-- (2.5,-2);
\draw (4,0)-- (5.5,-2);
\draw (5,0)-- (5.5,-2);
\draw (6,0)-- (5.5,-2);
\draw (7,0)-- (5.5,-2);
\draw (8,0)-- (8.5,-2);
\draw (9,0)-- (8.5,-2);
\draw (10,0)-- (8.5,-2);
\draw (11,0)-- (8.5,-2);

\draw (0,0)  [fill=black] circle (\vr);
\draw (1,0)  [fill=black] circle (\vr);
\draw (2,0)  [fill=black] circle (\vr);
\draw (3,0)  [fill=black] circle (\vr);
\draw (4,0)  [fill=black] circle (\vr);
\draw (5,0)  [fill=black] circle (\vr);
\draw (6,0)  [fill=black] circle (\vr);
\draw (7,0)  [fill=black] circle (\vr);
\draw (8,0)  [fill=black] circle (\vr);
\draw (9,0)  [fill=black] circle (\vr);
\draw (10,0)  [fill=black] circle (\vr);
\draw (11,0)  [fill=black] circle (\vr);

\draw (1,2)  [fill=black] circle (\vr);
\draw (4,2)  [fill=black] circle (\vr);
\draw (7,2)  [fill=black] circle (\vr);
\draw (10,2)  [fill=black] circle (\vr);
\draw (2.5,-2)  [fill=black] circle (\vr);
\draw (5.5,-2)  [fill=black] circle (\vr);
\draw (8.5,-2)  [fill=black] circle (\vr);

\draw (5.5,-3.5) node {$H_1$};

\end{tikzpicture}
\end{center}
\vspace{-0.6cm}
\caption{A graph $H_1\in \Gamma_1$.}\label{f1}\vspace{-0.1cm}
 \end{figure}
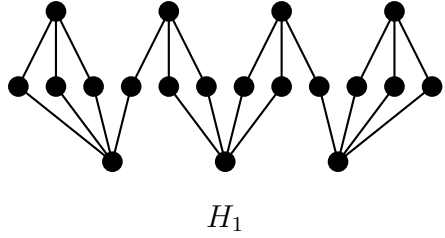

\vspace*{4mm}

Let $\Gamma_1$ be a class of connected graphs $H=(V,E)$ with $\ell\,(>0)$ edges $v_iv_j\in E(H)$, $(d_i,d_j)=(\Delta,\,\delta)$, and $m-\ell\,(>0)$ edges $v_iv_j\in E(H)$, $(d_i,d_j)=(\Delta-1,\,\delta)$, where $\Delta$, $\delta$ and $m$ are the maximum degree, the minimum degree, and the number of edges of graph $H$, respectively. A graph $H_1\in \Gamma_1$ (see, Fig. \ref{f1}). We now give a lower bound on $ISDD$ of graphs in terms of different graph parameters, and characterize the extremal graphs.
\begin{theorem} \label{aa2} Let $G$ be a graph with $m$ edges and the maximum degree $\Delta$ \& the minimum degree $\delta$. Then
   $$ISDD(G)\geq \frac{\Delta\,\delta}{\Delta^2+\delta^2}\,\ell+\frac{(\Delta-1)\,\delta}{(\Delta-1)^2+\delta^2}\,(m-\ell),$$
where $\ell\,(\geq 0)$ is the number of edges $v_iv_j\in E(G)$ such that $(d_i,\,d_j)=(\Delta,\,\delta)$. Moreover, if $G$ is connected, then the equality holds if and only if $G$ is a regular graph or $G$ is a $(\Delta,\,\delta)$-semiregular bipartite graph or $G\in \Gamma_1$.
\end{theorem}

\begin{proof} We divide the edges $E(G)$ in the following two classes:
\begin{align*}
&A=\Big\{v_iv_j\in E(G):\,(d_i,\,d_j)=(\Delta,\,\delta)\Big\},\\
~\mbox{ and }~&B=\Big\{v_iv_j\in E(G):\,(d_i,\,d_j)\neq (\Delta,\,\delta)\Big\},
\end{align*}
that is, $A\cup B=E(G)$ and $A\cap B=\emptyset$. Since $m$ is the number of edges, we have $|A|+|B|=m$. One can easily see that $|A|=\ell\geq 0$ and $|B|=m-\ell\geq 0$. Since $|A|=\ell$, we have
  $$\sum\limits_{v_iv_j\in E(G),\atop (d_i,\,d_j)=(\Delta,\,\delta)}\,\frac{d_i\,d_j}{d^2_i+d^2_j}=\frac{\Delta\,\delta}{\Delta^2+\delta^2}\,\ell.$$
By Lemma \ref{1a2}, we obtain
\begin{equation}
\sum\limits_{v_iv_j\in E(G),\atop (d_i,\,d_j)\neq (\Delta,\,\delta)}\,\frac{d_i\,d_j}{d^2_i+d^2_j}\geq \frac{(\Delta-1)\,\delta}{(\Delta-1)^2+\delta^2}\,(m-\ell)\label{1w2}
\end{equation}
with equality if and only if $(d_i,d_j)=(\Delta-1,\,\delta)$ for any edge $v_iv_j\in E(G)$ with $(d_i,\,d_j)\neq (\Delta,\,\delta)$. Using the above results, we obtain
\begin{align*}
ISDD(G)=\sum\limits_{v_iv_j\in E(G)}\,\frac{d_i\,d_j}{d^2_i+d^2_j}&=\sum\limits_{v_iv_j\in E(G),\atop (d_i,\,d_j)=(\Delta,\,\delta)}\,\frac{d_i\,d_j}{d^2_i+d^2_j}+\sum\limits_{v_iv_j\in E(G),\atop (d_i,\,d_j)\neq (\Delta,\,\delta)}\,\frac{d_i\,d_j}{d^2_i+d^2_j}\\[3mm]
&\geq \frac{\Delta\,\delta}{\Delta^2+\delta^2}\,\ell+\frac{(\Delta-1)\,\delta}{(\Delta-1)^2+\delta^2}\,(m-\ell).
\end{align*}
The first part of the proof is done.

\vspace*{3mm}

Suppose that equality holds. Then the equality holds in (\ref{1w2}). From the equality in (\ref{1w2}), we have $(d_i,d_j)=(\Delta-1,\,\delta)$ for any edge $v_iv_j\in B$. Thus we have $\ell$ edges $v_iv_j\in E(G)$ with $(d_i,d_j)=(\Delta,\,\delta)$, and $m-\ell$ edges $v_iv_j\in E(G)$ with $(d_i,\,d_j)=(\Delta-1,\,\delta)$. We have $m\geq \ell$. First we assume that $m=\ell$. Then $|B|=0$ and $|A|=m$, that is, all of the edges $v_iv_j\in E(G)$ with $(d_i,d_j)=(\Delta,\,\delta)$ in $G$. If $\Delta=\delta$, then $G$ is a regular graph. Otherwise, $\Delta\neq \delta$. Therefore, the vertices of degree $\Delta$ are adjacent to the vertices of degree $\delta$ and the vertices of degree $\delta$ are adjacent to the vertices of degree $\Delta$. Since $G$ is connected, one can easily see that $V(G)=U\cup W$, $U\cap W=\emptyset$ such that $v_iv_j\in E(G)$, $v_i\in U$, $v_j\in W$, where $U=\{v_k\in V(G):\,d_k=\Delta\}$ and $W=\{v_k\in V(G):\,d_k=\delta\}$. Hence $G$ is a $(\Delta,\,\delta)$-semiregular bipartite graph.

\vspace*{3mm}

Next we assume that $m>\ell\geq 0$. If $\ell=0$, then all the $m$ edges $v_iv_j\in E(G)$ with $(d_i,d_j)=(\Delta-1,\,\delta)$. This means that all the vertices are of degree either $\Delta-1$ or $\delta$, this is a contradiction as no vertex of degree $\Delta$. Otherwise, $m>\ell>0$. In this case $\ell\,(>0)$ edges $v_iv_j\in E(G)$ with $(d_i,d_j)=(\Delta,\,\delta)$ and $m-\ell\,(>0)$ edges $v_iv_j\in E(G)$ with $(d_i,\,d_j)=(\Delta-1,\,\delta)$. Since $G$ is connected, we have $G\in \Gamma_1$.

\vspace*{3mm}

Conversely, let $G$ be a regular graph. Then $m=\ell$ and $\displaystyle{\frac{d_i\,d_j}{d^2_i+d^2_j}}=\frac{1}{2}$ for any edge $v_iv_j\in E(G)$, and hence
  $$ISDD(G)=\frac{m}{2}=\frac{\Delta\,\delta}{\Delta^2+\delta^2}\,\ell+\frac{(\Delta-1)\,\delta}{(\Delta-1)^2+\delta^2}\,(m-\ell).$$

\vspace*{3mm}

Let $G$ be a $(\Delta,\,\delta)$-semiregular bipartite graph. Then $m=\ell$ and for any edge $v_iv_j\in E(G)$, we have $(d_i,d_j)=(\Delta,\,\delta)$, that is, $\displaystyle{\frac{d_i\,d_j}{d^2_i+d^2_j}}=\frac{\Delta\,\delta}{\Delta^2+\delta^2}$, and hence
 $$ISDD(G)=\frac{\Delta\,\delta}{\Delta^2+\delta^2}\,m=\frac{\Delta\,\delta}{\Delta^2+\delta^2}\,\ell+\frac{(\Delta-1)\,\delta}{(\Delta-1)^2+\delta^2}\,(m-\ell).$$

\vspace*{3mm}

Let $G\in \Gamma_1$. From the definition, we obtain
\begin{align*}
&\frac{d_i\,d_j}{d^2_i+d^2_j}=\frac{\Delta\,\delta}{\Delta^2+\delta^2}~\mbox{ for }v_iv_j\in E(G)~\mbox{ with }(d_i,d_j)=(\Delta,\,\delta)\\
   ~\mbox{ and }&\\
   &~\frac{d_i\,d_j}{d^2_i+d^2_j}=\frac{(\Delta-1)\,\delta}{(\Delta-1)^2+\delta^2}\mbox{ for }v_iv_j\in E(G)~\mbox{ with }(d_i,d_j)=(\Delta-1,\,\delta).
\end{align*}
Hence
$$ISDD(G)=\frac{\Delta\,\delta}{\Delta^2+\delta^2}\,\ell+\frac{(\Delta-1)\,\delta}{(\Delta-1)^2+\delta^2}\,(m-\ell).$$
\end{proof}

\begin{figure}[ht!]
\begin{center}
\begin{tikzpicture}[scale=0.5,style=thick]
\def\vr{7pt}
\def\br{0.6pt}

\draw (0,0)-- (3,3);
\draw (0,0)-- (11,3);
\draw (2,0)-- (3,3);
\draw (2,0)-- (11,3);
\draw (4,0)-- (3,3);
\draw (4,0)-- (7,3);
\draw (6,0)-- (3,3);
\draw (6,0)-- (7,3);
\draw (8,0)-- (3,3);
\draw (8,0)-- (11,3);
\draw (10,0)-- (11,3);
\draw (10,0)-- (15,3);
\draw (12,0)-- (7,3);
\draw (12,0)-- (15,3);
\draw (14,0)-- (11,3);
\draw (14,0)-- (15,3);
\draw (16,0)-- (7,3);
\draw (16,0)-- (15,3);
\draw (18,0)-- (7,3);
\draw (18,0)-- (15,3);

\draw (0,0)-- (2,0);
\draw (2,0)-- (4,0);
\draw (4,0)-- (6,0);
\draw (8,0)-- (10,0);
\draw (10,0)-- (12,0);
\draw (12,0)-- (14,0);
\draw (14,0)-- (16,0);
\draw (16,0)-- (18,0);
\draw (3,3)-- (7,3);
\draw (11,3)-- (15,3);

\draw (0,0)  [fill=black] circle (\vr);
\draw (2,0)  [fill=black] circle (\vr);
\draw (4,0)  [fill=black] circle (\vr);
\draw (6,0)  [fill=black] circle (\vr);
\draw (8,0)  [fill=black] circle (\vr);
\draw (10,0)  [fill=black] circle (\vr);
\draw (12,0)  [fill=black] circle (\vr);
\draw (14,0)  [fill=black] circle (\vr);
\draw (16,0)  [fill=black] circle (\vr);
\draw (18,0)  [fill=black] circle (\vr);
\draw (3,3)  [fill=black] circle (\vr);
\draw (7,3)  [fill=black] circle (\vr);
\draw (11,3)  [fill=black] circle (\vr);
\draw (15,3)  [fill=black] circle (\vr);

\draw[-] (0,0) to [bend right=30] (4,0);
\draw[-] (0,0) to [bend right=30] (6,0);
\draw[-] (2,0) to [bend right=30] (6,0);

\draw[-] (8,0) to [bend right=30] (14,0);
\draw[-] (10,0) to [bend right=30] (16,0);
\draw[-] (8,0) to [bend right=30] (18,0);
\draw[-] (12,0) to [bend right=30] (18,0);

\draw (8,-3.5) node {$H_2$};

\end{tikzpicture}
\end{center}
\vspace{-0.6cm}
\caption{A graph $H_2\in \Gamma_2$.}\label{f2}\vspace{-0.7cm}
 \end{figure}

\vspace*{3mm}

The following result is obtained from Lemma 2.5 in \cite{RSD1}.
\begin{lemma} {\rm \cite{RSD1}} \label{pg1} Let $v_iv_j\in E(G)$ be any edge in $G$ with $d_i>d_j$. Then
   $$\frac{d_i\,d_j}{d^2_i+d^2_j}\leq \frac{\Delta\,(\Delta-1)}{\Delta^2+(\Delta-1)^2},$$
where $\Delta$ is the maximum degree in $G$. Moreover, the equality holds if and only if $d_i=\Delta$ and $d_j=\Delta-1$.
\end{lemma}

Let $\Gamma_2$ be a class of connected graphs $H=(V,E)$ with $k\,(>0)$ edges $v_iv_j\in E(H)$ with $d_i=d_j=\Delta$ or  $d_i=d_j=\Delta-1$, and $m-k\,(m>k)$ edges $v_iv_j\in E(H)$ with $d_i=\Delta,\,d_j=\Delta-1$. A graph $H_2\in \Gamma_2$ (see, Fig. \ref{f2}). We now give an upper bound on $ISDD$ of graphs in terms of various graph parameters, and characterize the extremal graphs.
\begin{theorem} \label{aa2} Let $G$ be a graph with $m$ edges and the maximum degree $\Delta$. Then
   $$ISDD(G)\leq \frac{k}{2}+\frac{\Delta\,(\Delta-1)}{\Delta^2+(\Delta-1)^2}\,(m-k),$$
where $k$ is the number of edges $v_iv_j\in E(G)$ such that $d_i=d_j$. Moreover, if $G$ is connected, then the equality holds if and only if $G$ is a regular graph or $G$ is a $(\Delta,\,\Delta-1)$-semiregular bipartite graph or $G\in \Gamma_2$.
\end{theorem}

\begin{proof} We divide the edges $E(G)$ in the following two classes:
\begin{align*}
C=\Big\{v_iv_j\in E(G):\,d_i=d_j\Big\}~\mbox{ and }~D=\Big\{v_iv_j\in E(G):\,d_i\neq d_j\Big\},
\end{align*}
that is, $C\cup D=E(G)$ and $C\cap D=\emptyset$. Since $m$ is the number of edges in $G$, we have $|C|+|D|=m$. One can easily see that $|C|=k\geq 0$ and $|D|\geq 0$. Since $|C|+|D|=m$, we have $|D|=m-k$. We have
  $$\sum\limits_{v_iv_j\in E(G),\atop d_i=d_j}\,\frac{d_i\,d_j}{d^2_i+d^2_j}=\frac{k}{2}.$$
From Lemma \ref{pg1}, we have
\begin{equation}
\sum\limits_{v_iv_j\in E(G),\atop d_i\neq d_j}\,\frac{d_i\,d_j}{d^2_i+d^2_j}\leq \frac{\Delta\,(\Delta-1)}{\Delta^2+(\Delta-1)^2}\,(m-k)\label{p1w1}
\end{equation}
with equality if and only if $(d_i,d_j)=(\Delta,\,\Delta-1)$ for any edge $v_iv_j\in E(G)$ with $d_i\neq d_j$. Now,
\begin{align*}
ISDD(G)=\sum\limits_{v_iv_j\in E(G)}\,\frac{d_i\,d_j}{d^2_i+d^2_j}&=\sum\limits_{v_iv_j\in E(G),\atop d_i=d_j}\,\frac{d_i\,d_j}{d^2_i+d^2_j}+\sum\limits_{v_iv_j\in E(G),\atop d_i\neq d_j}\,\frac{d_i\,d_j}{d^2_i+d^2_j}\\[3mm]
&\leq \frac{k}{2}+\frac{\Delta\,(\Delta-1)}{\Delta^2+(\Delta-1)^2}\,(m-k).
\end{align*}
The first part of the proof is done.

\vspace*{3mm}

Suppose that equality holds. Then the equality holds in (\ref{p1w1}). From the equality in (\ref{p1w1}), we have $(d_i,d_j)=(\Delta,\,\Delta-1)$ for any edge $v_iv_j\in E(G)$ with $d_i\neq d_j$. Thus we have $k$ edges $v_iv_j\in E(G)$ with $d_i=d_j$, and the remaining $m-k$ edges $v_iv_j\in E(G)$ with $(d_i,\,d_j)=(\Delta,\,\Delta-1)$. We have $m\geq k\geq 0$. If $m=k$, then $G$ is a regular graph as $G$ is connected. Otherwise, $m>k\geq 0$. First we assume that $k=0$. Then there are any edge $v_iv_j\in E(G)$ with $(d_i,\,d_j)=(\Delta,\,\Delta-1)$. Therefore, the vertices of degree $\Delta$ is adjacent to the vertices of degree $\Delta-1$ and the vertices of degree $\Delta-1$ is adjacent to the vertices of degree $\Delta$. Since $G$ is connected, one can easily see that $V(G)=U\cup W$ with $U\cap W=\emptyset$ such that $v_iv_j\in E(G)$ with $v_i\in U$, $v_j\in W$, where $U=\{v_k\in V(G):\,d_k=\Delta\}$ and $W=\{v_k\in V(G):\,d_k=\Delta-1\}$. Hence $G$ is a $(\Delta,\,\Delta-1)$-semiregular bipartite graph.

\vspace*{3mm}

Next we assume that $m>k>0$. Since $G$ is connected and $m>k>0$, then $m-k\,(>0)$ edges $v_iv_j\in E(G)$ with $(d_i,\,d_j)=(\Delta,\,\Delta-1)$ and the remaining $k\,(>0)$ edges $v_iv_j\in E(G)$ such that $d_i=d_j=\Delta$ or  $d_i=d_j=\Delta-1$. Hence $G\in \Gamma_2$.

\vspace*{3mm}

Conversely, let $G$ be a regular graph. Then $m=k$ and $\displaystyle{\frac{d_i\,d_j}{d^2_i+d^2_j}}=\frac{1}{2}$ for any edge $v_iv_j\in E(G)$, and hence
  $$ISDD(G)=\frac{m}{2}=\frac{k}{2}+\frac{\Delta\,(\Delta-1)}{\Delta^2+(\Delta-1)^2}\,(m-k).$$

\vspace*{3mm}

Let $G$ be a $(\Delta,\,\Delta-1)$-semiregular bipartite graph. Then $k=0$ and for any edge $v_iv_j\in E(G)$, we have $(d_i,d_j)=(\Delta,\,\Delta-1)$, that is, $\displaystyle{\frac{d_i\,d_j}{d^2_i+d^2_j}}=\frac{\Delta\,(\Delta-1)}{\Delta^2+(\Delta-1)^2}$, and hence
 $$ISDD(G)=\displaystyle{\frac{\Delta\,(\Delta-1)}{\Delta^2+(\Delta-1)^2}}\,m=\frac{k}{2}+\frac{\Delta\,(\Delta-1)}{\Delta^2+(\Delta-1)^2}\,(m-k).$$

\vspace*{3mm}

Let $G\in \Gamma_2$. Then $k>0$ edges $v_iv_j\in E(G)$ with $d_i=d_j=\Delta$ or $d_i=d_j=\Delta-1$, and $m-k\,(m>k)$ edges $v_iv_j\in E(G)$ with $(d_i,\,d_j)=(\Delta,\,\Delta-1)$. Thus we have
   $$\sum\limits_{v_iv_j\in E(G),\atop d_i=d_j}\,\frac{d_i\,d_j}{d^2_i+d^2_j}=\frac{k}{2}~\mbox{ and }~\sum\limits_{v_iv_j\in E(G),\atop d_i\neq d_j}\,\frac{d_i\,d_j}{d^2_i+d^2_j}=\frac{\Delta\,(\Delta-1)}{\Delta^2+(\Delta-1)^2}\,(m-k).$$
Hence
$$ISDD(G)=\sum\limits_{v_iv_j\in E(G),\atop d_i=d_j}\,\frac{d_i\,d_j}{d^2_i+d^2_j}+\sum\limits_{v_iv_j\in E(G),\atop d_i\neq d_j}\,\frac{d_i\,d_j}{d^2_i+d^2_j}=\frac{k}{2}+\frac{\Delta\,(\Delta-1)}{\Delta^2+(\Delta-1)^2}\,(m-k).$$
\end{proof}

\begin{corollary} \label{p1a2} Let $G$ be a graph of order $n$ with maximum degree $\Delta$. Then
\begin{equation}
ISDD(G)\leq \frac{k}{2}+\frac{\Delta\,(\Delta-1)}{\Delta^2+(\Delta-1)^2}\,\left(\frac{n\,\Delta}{2}-k\right),\label{g2}
\end{equation}
where $k$ is the number of edges $v_iv_j\in E(G)$ such that $d_i=d_j$. Moreover, if $G$ is connected, then the equality holds in (\ref{g2}) if and only if $G$ is a regular graph.
\end{corollary}

\begin{proof} Since $2m\leq n\,\Delta$, from Theorem \ref{aa2}, we get the result. We have $2m=n\,\Delta$ if and only if $G$ is a regular graph. Moreover, if $G$ is connected, then the equality holds in (\ref{g2}) if and only if $G$ is a regular graph, by Theorem \ref{aa2}.
\end{proof}

\section{Relation between $ISDD$ index with some topological indices of graphs}

In this section we obtain some relations between $ISDD$ index with some other topological indices of graphs. First we give a well-known relation between $ISDD$ and $GA$ of graphs in the following:
\begin{lemma}{\rm \cite{M1}} \label{1e1} Let $G$ be graph with $m$ edges. Then $ISDD(G)\geq \frac{GA^2(G)}{4m}$.
\end{lemma}

We now mention an inequality that will be used to produce the main result in this section.
 \begin{lemma}{\rm \cite{RA}} {\rm (Radon's inequality)} \label{k1} If $a_k,\,x_k>0$, $p>0$, $k\in\{1,\,2,\ldots,\,r\}$, then the following inequality holds:
          $$\sum\limits^r_{k=1}\,\frac{x^{p+1}_k}{a^p_k}\geq \frac{\left(\sum\limits^r_{k=1}\,x_k\right)^{p+1}}{\left(\sum\limits^r_{k=1}\,a_k\right)^p}$$
 with equality if and only if
 $\frac{x_1}{a_1}=\frac{x_2}{a_2}=\cdots=\frac{x_r}{a_r}$.
 \end{lemma}

We now obtain a relation between $ISDD$, $GA$ and $M_2$ of graphs.
\begin{theorem} Let $G$ be graph with $m$ edges and maximum degree $\Delta$. Then
\begin{equation}
ISDD(G)\geq \displaystyle{\frac{\Delta^2\,GA^2(G)}{4m\,\Delta^2-2\,M_2(G)}}\label{1e0}
\end{equation}
with equality if and only if $G$ is a regular graph.
\end{theorem}

\begin{proof} Setting $x_k=\displaystyle{\frac{\sqrt{d_id_j}}{d_i+d_j}}$ and $a_k=\displaystyle{\frac{d^2_i+d^2_j}{(d_i+d_j)^2}}$ in Lemma \ref{k1}, we obtain
\begin{eqnarray*}
\frac{\left(\sum\limits_{v_iv_j\in E(G)}\,\displaystyle{\frac{\sqrt{d_id_j}}{d_i+d_j}}\right)^2}{\sum\limits_{v_iv_j\in E(G)}\,\displaystyle{\frac{d^2_i+d^2_j}{(d_i+d_j)^2}}}\leq \sum\limits_{v_iv_j\in E(G)}\,\displaystyle{\frac{d_i\,d_j}{d^2_i+d^2_j}},
\end{eqnarray*}
that is,
\begin{equation}
ISDD(G)\geq \frac{GA^2(G)}{4\,A},~~\mbox{ where }~A=\sum\limits_{v_iv_j\in E(G)}\,\displaystyle{\frac{d^2_i+d^2_j}{(d_i+d_j)^2}}.\label{1e2}
\end{equation}
By Lemma \ref{k1}, the equality holds if and only if
\begin{equation}
\frac{\sqrt{d_i\,d_j}\,(d_i+d_j)}{d^2_i+d^2_j}=\frac{\sqrt{d_i\,d_k}\,(d_i+d_k)}{d^2_i+d^2_k}~~\mbox{ for each edge }v_iv_j,\,v_iv_k\in E(G).\label{1e3}
\end{equation}
Since $d_i\leq \Delta$ for all $v_i\in V(G)$, we obtain
\begin{equation}
\sum\limits_{v_iv_j\in E(G)}\,\displaystyle{\frac{2d_i\,d_j}{(d_i+d_j)^2}}\geq \frac{1}{2\,\Delta^2}\,\sum\limits_{v_iv_j\in E(G)}\,d_i\,d_j=\frac{1}{2\,\Delta^2}\,M_2(G)\label{1e4}
\end{equation}
with equality if and only if $d_i=\Delta$ for all $v_i\in V(G)$, that is, if and only if $G$ is a regular graph. Using the above result, we obtain
\begin{align*}
A=\sum\limits_{v_iv_j\in E(G)}\,\displaystyle{\frac{d^2_i+d^2_j}{(d_i+d_j)^2}}&=\sum\limits_{v_iv_j\in E(G)}\,\left[1-\displaystyle{\frac{2d_i\,d_j}{(d_i+d_j)^2}}\right]\\[3mm]
&=m-\sum\limits_{v_iv_j\in E(G)}\,\displaystyle{\frac{2d_i\,d_j}{(d_i+d_j)^2}}\leq m-\frac{1}{2\,\Delta^2}\,M_2(G).
\end{align*}
Using this in (\ref{1e2}), we obtain
   $$ISDD(G)\geq \frac{GA^2(G)}{4\,A}\geq \frac{GA^2(G)}{4\,m-\frac{2}{\Delta^2}\,M_2(G)}.$$
The first part of the proof is done.

\vspace*{3mm}

Suppose that equality holds. Then the equality holds in (\ref{1e4}). Thus $G$ is a regular graph.

\vspace*{3mm}

Conversely, let $G$ be a $r$-regular graph. Then we have $\Delta=r$, $ISDD(G)=\frac{1}{2}\,m$, $GA(G)=m$ and $M_2(G)=m\,r^2$. Therefore the equality holds in (\ref{1e0}).
\end{proof}

\begin{remark} \label{1e5} Our result in (\ref{1e0}) is always better than the result in Lemma \ref{1e1}. Since $M_2(G)\geq 0$, we obtain
$$ISDD(G)\geq \displaystyle{\frac{\Delta^2\,GA^2(G)}{4m\,\Delta^2-2\,M_2(G)}}\geq \displaystyle{\frac{GA^2(G)}{4m}}.$$
\end{remark}

\begin{remark} Indeed the inequality in Lemma \ref{1e1} is strict. Since $G$ is non-empty graph, $M_2(G)>0$. From the above Remark \ref{1e5}, we obtain
$$ISDD(G)\geq \displaystyle{\frac{\Delta^2\,GA^2(G)}{4m\,\Delta^2-2\,M_2(G)}}>\displaystyle{\frac{GA^2(G)}{4m}}.$$
\end{remark}

\vspace*{3mm}

\begin{figure}[ht!]
\begin{center}
\begin{tikzpicture}[scale=0.8,style=thick]
\def\vr{5pt}
\def\br{0.5pt}

\draw (9,0)-- (2,4);
\draw (9,0)-- (4,4);
\draw (9,0)-- (6,4);
\draw (9,0)-- (8,4);
\draw (9,0)--(10,4);
\draw (9,0)-- (12,4);
\draw (9,0)-- (14,4);
\draw (9,0)-- (16,4);
\draw (9,0)-- (18,4);

\draw (10,0)-- (2,4);
\draw (10,0)-- (4,4);
\draw (10,0)-- (6,4);
\draw (10,0)-- (8,4);
\draw (10,0)--(10,4);
\draw (10,0)-- (12,4);
\draw (10,0)-- (14,4);
\draw (10,0)-- (16,4);
\draw (10,0)-- (18,4);

\draw (11,0)-- (2,4);
\draw (11,0)-- (4,4);
\draw (11,0)-- (6,4);
\draw (11,0)-- (8,4);
\draw (11,0)--(10,4);
\draw (11,0)-- (12,4);
\draw (11,0)-- (14,4);
\draw (11,0)-- (16,4);
\draw (11,0)-- (18,4);

\draw (12,0)-- (2,4);
\draw (12,0)-- (4,4);
\draw (12,0)-- (6,4);
\draw (12,0)-- (8,4);
\draw (12,0)--(10,4);
\draw (12,0)-- (12,4);
\draw (12,0)-- (14,4);
\draw (12,0)-- (16,4);
\draw (12,0)-- (18,4);

\draw (13,0)-- (2,4);
\draw (13,0)-- (4,4);
\draw (13,0)-- (6,4);
\draw (13,0)-- (8,4);
\draw (13,0)--(10,4);
\draw (13,0)-- (12,4);
\draw (13,0)-- (14,4);
\draw (13,0)-- (16,4);
\draw (13,0)-- (18,4);

\draw (14,0)-- (2,4);
\draw (14,0)-- (4,4);
\draw (14,0)-- (6,4);
\draw (14,0)-- (8,4);
\draw (14,0)--(10,4);
\draw (14,0)-- (12,4);
\draw (14,0)-- (14,4);
\draw (14,0)-- (16,4);
\draw (14,0)-- (18,4);

\draw (15,0)-- (2,4);
\draw (15,0)-- (4,4);
\draw (15,0)-- (6,4);
\draw (15,0)-- (8,4);
\draw (15,0)--(10,4);
\draw (15,0)-- (12,4);
\draw (15,0)-- (14,4);
\draw (15,0)-- (16,4);
\draw (15,0)-- (18,4);

\draw (16,0)-- (2,4);
\draw (16,0)-- (4,4);
\draw (16,0)-- (6,4);
\draw (16,0)-- (8,4);
\draw (16,0)--(10,4);
\draw (16,0)-- (12,4);
\draw (16,0)-- (14,4);
\draw (16,0)-- (16,4);
\draw (16,0)-- (18,4);

\draw (17,0)-- (2,4);
\draw (17,0)-- (4,4);
\draw (17,0)-- (6,4);
\draw (17,0)-- (8,4);
\draw (17,0)--(10,4);
\draw (17,0)-- (12,4);
\draw (17,0)-- (14,4);
\draw (17,0)-- (16,4);
\draw (17,0)-- (18,4);

\draw (18,0)-- (2,4);
\draw (18,0)-- (4,4);
\draw (18,0)-- (6,4);
\draw (18,0)-- (8,4);
\draw (18,0)--(10,4);
\draw (18,0)-- (12,4);
\draw (18,0)-- (14,4);
\draw (18,0)-- (16,4);
\draw (18,0)-- (18,4);

\draw (19,0)-- (2,4);
\draw (19,0)-- (4,4);
\draw (19,0)-- (6,4);
\draw (19,0)-- (8,4);
\draw (19,0)--(10,4);
\draw (19,0)-- (12,4);
\draw (19,0)-- (14,4);
\draw (19,0)-- (16,4);
\draw (19,0)-- (18,4);

\draw (20,0)-- (2,4);
\draw (20,0)-- (4,4);
\draw (20,0)-- (6,4);
\draw (20,0)-- (8,4);
\draw (20,0)--(10,4);
\draw (20,0)-- (12,4);
\draw (20,0)-- (14,4);
\draw (20,0)-- (16,4);
\draw (20,0)-- (18,4);

\draw (0,0)-- (2,4);
\draw (0,0)-- (4,4);
\draw (0,0)-- (6,4);
\draw (0,0)-- (8,4);
\draw (0,0)--(10,4);
\draw (0,0)-- (12,4);

\draw (1,0)-- (4,4);
\draw (1,0)-- (6,4);
\draw (1,0)-- (8,4);
\draw (1,0)--(10,4);
\draw (1,0)-- (12,4);
\draw (1,0)-- (14,4);

\draw (2,0)-- (6,4);
\draw (2,0)-- (8,4);
\draw (2,0)--(10,4);
\draw (2,0)-- (12,4);
\draw (2,0)-- (14,4);
\draw (2,0)-- (16,4);

\draw (3,0)-- (8,4);
\draw (3,0)--(10,4);
\draw (3,0)-- (12,4);
\draw (3,0)-- (14,4);
\draw (3,0)-- (16,4);
\draw (3,0)-- (18,4);

\draw (4,0)--(10,4);
\draw (4,0)-- (12,4);
\draw (4,0)-- (14,4);
\draw (4,0)-- (16,4);
\draw (4,0)-- (18,4);
\draw (4,0)-- (2,4);

\draw (5,0)-- (12,4);
\draw (5,0)-- (14,4);
\draw (5,0)-- (16,4);
\draw (5,0)-- (18,4);
\draw (5,0)-- (2,4);
\draw (5,0)--(4,4);

\draw (6,0)-- (14,4);
\draw (6,0)-- (16,4);
\draw (6,0)-- (18,4);
\draw (6,0)-- (2,4);
\draw (6,0)--(4,4);
\draw (6,0)-- (6,4);

\draw (7,0)-- (16,4);
\draw (7,0)-- (18,4);
\draw (7,0)-- (2,4);
\draw (7,0)--(4,4);
\draw (7,0)-- (6,4);
\draw (7,0)-- (8,4);

\draw (8,0)-- (18,4);
\draw (8,0)-- (2,4);
\draw (8,0)--(4,4);
\draw (8,0)-- (6,4);
\draw (8,0)-- (8,4);
\draw (8,0)-- (10,4);

\draw (0,0)  [fill=black] circle (\vr);
\draw (1,0)  [fill=black] circle (\vr);
\draw (2,0)  [fill=black] circle (\vr);
\draw (3,0)  [fill=black] circle (\vr);
\draw (4,0)  [fill=black] circle (\vr);
\draw (5,0)  [fill=black] circle (\vr);
\draw (6,0)  [fill=black] circle (\vr);
\draw (7,0)  [fill=black] circle (\vr);
\draw (8,0)  [fill=black] circle (\vr);
\draw (9,0)  [fill=black] circle (\vr);
\draw (10,0)  [fill=black] circle (\vr);
\draw (11,0)  [fill=black] circle (\vr);
\draw (12,0)  [fill=black] circle (\vr);
\draw (13,0)  [fill=black] circle (\vr);
\draw (14,0)  [fill=black] circle (\vr);
\draw (15,0)  [fill=black] circle (\vr);
\draw (16,0)  [fill=black] circle (\vr);
\draw (17,0)  [fill=black] circle (\vr);
\draw (18,0)  [fill=black] circle (\vr);
\draw (19,0)  [fill=black] circle (\vr);
\draw (20,0)  [fill=black] circle (\vr);

\draw (2,4)  [fill=black] circle (\vr);
\draw (4,4)  [fill=black] circle (\vr);
\draw (6,4)  [fill=black] circle (\vr);
\draw (8,4)  [fill=black] circle (\vr);
\draw (10,4)  [fill=black] circle (\vr);
\draw (12,4)  [fill=black] circle (\vr);
\draw (14,4)  [fill=black] circle (\vr);
\draw (16,4)  [fill=black] circle (\vr);
\draw (18,4)  [fill=black] circle (\vr);

\draw (0,-0.5) node {$v_1$};
\draw (1,-0.5) node {$v_2$};
\draw (2,-0.5) node {$v_3$};
\draw (3,-0.5) node {$v_4$};
\draw (4,-0.5) node {$v_5$};
\draw (5,-0.5) node {$v_6$};
\draw (6,-0.5) node {$v_7$};
\draw (7,-0.5) node {$v_8$};
\draw (8,-0.5) node {$v_9$};
\draw (9,-0.5) node {$v_{10}$};
\draw (10,-0.5) node {$v_{11}$};
\draw (11,-0.5) node {$v_{12}$};
\draw (12,-0.5) node {$v_{13}$};
\draw (13,-0.5) node {$v_{14}$};
\draw (14,-0.5) node {$v_{15}$};
\draw (15,-0.5) node {$v_{16}$};
\draw (16,-0.5) node {$v_{17}$};
\draw (17,-0.5) node {$v_{18}$};
\draw (18,-0.5) node {$v_{19}$};
\draw (19,-0.5) node {$v_{20}$};
\draw (20,-0.5) node {$v_{21}$};

\draw (2,4.5) node {$u_1$};
\draw (4,4.5) node {$u_2$};
\draw (6,4.5) node {$u_3$};
\draw (8,4.5) node {$u_4$};
\draw (10,4.5) node {$u_5$};
\draw (12,4.5) node {$u_6$};
\draw (14,4.5) node {$u_7$};
\draw (16,4.5) node {$u_8$};
\draw (18,4.5) node {$u_9$};

\draw (10,-2) node {$H_3$};
\end{tikzpicture}
\end{center}
\vspace{-0.1cm}
\caption{A graph $H_3\in \Gamma_3$.}\label{FIG1}\vspace{-0.2cm}
 \end{figure}

 \vspace*{10mm}

In Fig. \ref{FIG1}, $H_3=(V,E)$, where $V(H_3)=U\cup W$, $U=\{u_1,\,u_2,\,u_3,\,u_4,\,u_5,\,u_6,\,u_7,\,u_8,\,u_9\}$, $W=\{v_1,\,v_2,\,v_3,\,v_4,\,v_5,\,v_6,\,v_7,\,v_8,\,v_9,\,v_{10},\,v_{11},\,v_{12},\,v_{13},\,v_{14},\,v_{15},\,v_{16},\,v_{17},\,v_{18},\,v_{19},\,v_{20},\,v_{21}\}$, $E(H_3)=E_1\cup E_2$, $E_1=\{v_iu_j:\,1\leq i\leq 9,\,i\leq j\leq i+5\leq 9,~\mbox{ and }~1\leq j\leq i-4~\mbox{ if }~i\geq 5\}$, $E_2=\{v_iu_j:\,10\leq i\leq 21,\,1\leq j\leq 9\}$. Moreover, each vertex degree from $u_1$ to $u_9$ is $18$, each vertex degree from $v_1$ to $v_9$ is $6$ and each vertex degree from $v_{10}$ to $v_{21}$ is $9$.
Let $\Gamma_3$ be a class of connected bipartite graphs $H=(V,E)$ with bipartition $V(H)=U\cup W$ such that
$$d_i=\left\{
                                     \begin{array}{ll}
                                      \Delta  & \hbox{if $v_i\in U$,} \\
                                       \delta~\mbox{ or }~\displaystyle{\frac{\Delta\,(\Delta-\delta)}{\Delta+\delta}} & \hbox{if $v_i\in W$.}
                                     \end{array}
                                   \right.$$
A graph $H_3\in \Gamma_3$.

\begin{lemma} \label{1q1} Let $G$ be a connected graph. Then $$\displaystyle{\frac{d_i+d_j}{d^2_i+d^2_j}}=\displaystyle{\frac{d_k+d_{\ell}}{d^2_k+d^2_{\ell}}}~\mbox{ for any two edges }v_iv_j,\,v_kv_{\ell}\in E(G)$$
if and only if $G$ is a regular graph or $G$ is a semiregular bipartite graph or $G\in \Gamma_3$.
\end{lemma}

\begin{proof} First we assume that
$$\displaystyle{\frac{d_i+d_j}{d^2_i+d^2_j}}=\displaystyle{\frac{d_k+d_{\ell}}{d^2_k+d^2_{\ell}}}~\mbox{ for any two edges }v_iv_j,\,v_kv_{\ell}\in E(G).$$
Since $G$ is a connected graph, for any edges $v_iv_j,\,v_iv_k\in E(G)$, we obtain
 $$\displaystyle{\frac{d_i+d_j}{d^2_i+d^2_j}}=\displaystyle{\frac{d_i+d_k}{d^2_i+d^2_k}},$$
that is,
 $$(d_i+d_j)\,(d^2_i+d^2_k)=(d_i+d_k)\,(d^2_i+d^2_j),$$
that is,
 $$(d_j-d_k)\,\Big[d^2_i-d_i\,(d_j+d_k)-d_jd_k\Big]=0,$$
that is,
\begin{equation}
\mbox{either $d_j=d_k$ or $d_i=d_j+d_k+\frac{d_jd_k}{d_i}$}.\label{1w1}
\end{equation}

Let $U=\{v_i\in V(G)\,:\,d_i=\Delta\}$, where $\Delta$ is the maximum degree in $G$. Then $|U|\geq 1$. Also let $E^{\prime}=\{v_iv_j\in E(G):\,v_i,\,v_j\in U\}$. Then we have $E^{\prime}\subseteq E(G)$ and $|E^{\prime}|\geq 0$. We consider the following cases:

\vspace*{3mm}

\noindent
${\bf Case\,1.}$ $|E^{\prime}|\geq 1$. Let $v_iv_j$ be an edge in $E^{\prime}$. Then $d_i=d_j=\Delta$. By (\ref{1w1}), we conclude that all the vertices in $N_G(v_i)$ $(\mbox{and }N_G(v_j))$ are of degree $\Delta$ as $d_i\neq d_j+d_k+\frac{d_jd_k}{d_i}$. Thus we have $N_G(v_i)\subseteq U$ and $N_G(v_j)\subseteq U$. By a similar argument, we prove that $N_G(v_k)\subseteq U$, where $v_k\in N_G(v_i)$. Continuing the procedure, it is easy to see, since $G$ is connected, that $V(G)=U$ and hence $G$ is a regular graph.

\vspace*{3mm}

\noindent
${\bf Case\,2.}$ $|E^{\prime}|=0$. Let $W=V(G)\backslash U$. In this case $|W|\geq 1$ as $G$ is connected. Let $v_i$ be any vertex in $U$. Then $N_G(v_i)\subseteq W$ as $|E^{\prime}|=0$. For any vertex $v_j\in N_G(v_i)$ with $v_i\in U$, by (\ref{1w1}), we obtain $N_G(v_j)\subseteq U$ as $d_j<\Delta+d_k+\frac{\Delta\,d_k}{d_j}$. By a similar argument, one can easily see that $N_{N_G(v_i)}\subseteq U$ and $N_{N_G(v_j)}\subseteq W$ (where $N_{N_G(v_j)}$ is the neighbor of neighbor set of vertex
$v_j$). Since $G$ is connected, using the same process, it is clear that $V(G)=U\cup W$ and the subgraphs that are induced by $U$ and $W$, respectively, are empty graphs. Hence $G$ is a bipartite graph. If all the vertices in $W$ have same degree $b$, (say), then $G$ is a $(\Delta,b)$-semiregular bipartite graph. Otherwise, all the vertices have not same degree in $W$. Let $v_n$ be a minimum degree vertex of degree $\delta$ in $G$. Then $v_n\in W$. Since $G$ is connected and all the vertices have not same degree in $W$, then there exists a vertex $v_s$ in $U$ such that $v_sv_n\in E(G)$ and $v_sv_t\in E(G)$, where $v_t\in W$, $d_t>\delta$. Then by (\ref{1w1}), we obtain
  $$\Delta=\delta+d_t+\frac{\delta\,d_t}{\Delta},~\mbox{ that is, }~d_t=\frac{\Delta\,(\Delta-\delta)}{\Delta+\delta}.$$
Hence $G\in \Gamma_3$.

\vspace*{3mm}

Conversely, let $G$ be an $r$-regular graph. Then one can easily check that
$$\displaystyle{\frac{d_i+d_j}{d^2_i+d^2_j}}=\displaystyle{\frac{d_k+d_{\ell}}{d^2_k+d^2_{\ell}}}=\frac{1}{r}~\mbox{ for any two edges }v_iv_j,\,v_kv_{\ell}\in E(G).$$

Let $G$ be a $(\Delta,\delta)$-semiregular bipartite graph. Then one can easily check that
$$\displaystyle{\frac{d_i+d_j}{d^2_i+d^2_j}}=\displaystyle{\frac{d_k+d_{\ell}}{d^2_k+d^2_{\ell}}}=\displaystyle{\frac{\Delta+\delta}{\Delta^2+\delta^2}}~\mbox{ for any two edges }v_iv_j,\,v_kv_{\ell}\in E(G).$$

Let $G\in \Gamma_3$. Then $G$ is a connected bipartite graph with $(d_i,\,d_j)=(\Delta,\,\delta)$ or $(d_i,\,d_j)=\left(\Delta,\,\displaystyle{\frac{\Delta\,(\Delta-\delta)}{\Delta+\delta}}\right)$ for any edge $v_iv_j\in E(G)$.

For $v_iv_j\in E(G)$ with $(d_i,\,d_j)=(\Delta,\,\delta)$, we obtain
 $$\displaystyle{\frac{d_i+d_j}{d^2_i+d^2_j}}=\displaystyle{\frac{\Delta+\delta}{\Delta^2+\delta^2}}.$$

For $v_iv_j\in E(G)$ with $(d_i,\,d_j)=\left(\Delta,\,\displaystyle{\frac{\Delta\,(\Delta-\delta)}{\Delta+\delta}}\right)$, we obtain
 $$\displaystyle{\frac{d_i+d_j}{d^2_i+d^2_j}}=\frac{\Delta+\displaystyle{\frac{\Delta\,(\Delta-\delta)}{\Delta+\delta}}}{\Delta^2+\displaystyle{\frac{\Delta^2\,(\Delta-\delta)^2}{(\Delta+\delta)^2}}}=\displaystyle{\frac{\Delta+\delta}{\Delta^2+\delta^2}}.$$
Hence
$$\displaystyle{\frac{d_i+d_j}{d^2_i+d^2_j}}=\displaystyle{\frac{d_k+d_{\ell}}{d^2_k+d^2_{\ell}}}=\displaystyle{\frac{\Delta+\delta}{\Delta^2+\delta^2}}~\mbox{ for any two edges }v_iv_j,\,v_kv_{\ell}\in E(G).$$

\vspace*{3mm}

\noindent
This completes the proof of the result.
\end{proof}

\begin{lemma} \label{t1sv1}
 The weighted $AM$-$HM$ inequality relates the weighted arithmetic and harmonic means. It states that for any list of weights $\omega_1,\, \omega_2, \ldots,\,\omega_n \geq 0$ such that $\omega_1 + \omega_2 + \cdots + \omega_n = \omega$,
   $$\frac{\omega_1\,x_1+\omega_2\,x_2+\cdots+\omega_n\,x_n}{\omega}\geq \frac{\omega}{\frac{\omega_1}{x_1}+\frac{\omega_2}{x_2}+\cdots+\frac{\omega_n}{x_n}}$$
 with equality if and only if $x_1=x_2=\cdots=x_n$.
 \end{lemma}
We now present a relation between $ISDD$ with the first Zagreb index $M_1(G)$ and the forgotten topological index $F(G)$ of graphs.
\begin{theorem} Let $G$ be a graph with $m$ edges. Then
 $$ISDD(G)\geq \frac{M_1(G)^2}{2\,F(G)}-\frac{m}{2}$$
with equality if and only if $G$ is a regular graph or $G$ is a semiregular bipartite graph or $G\in \Gamma_3$.
\end{theorem}

\begin{proof} Setting $\omega_k=d_i+d_j$ and $x_k=\displaystyle{\frac{d_i+d_j}{d^2_i+d^2_j}}$ in Lemma \ref{t1sv1}, we obtain
  $$\frac{\displaystyle{\sum\limits_{v_iv_j\in E(G)}\,\frac{(d_i+d_j)^2}{d^2_i+d^2_j}}}{\sum\limits_{v_iv_j\in E(G)}\,(d_i+d_j)}\geq \frac{\sum\limits_{v_iv_j\in E(G)}\,(d_i+d_j)}{\sum\limits_{v_iv_j\in E(G)}\,(d^2_i+d^2_j)},$$
that is,
  $$\sum\limits_{v_iv_j\in E(G)}\,\frac{(d_i+d_j)^2}{d^2_i+d^2_j}\geq \frac{M_1(G)^2}{F(G)},$$
where
\begin{align*}
&M_1(G)=\sum\limits_{v_i\in V(G)}\,d^2_i=\sum\limits_{v_iv_j\in E(G)}\,(d_i+d_j)\\[3mm]
~\mbox{ and }~&F(G)=\sum\limits_{v_i\in V(G)}\,d^3_i=\sum\limits_{v_iv_j\in E(G)}\,(d^2_i+d^2_j).
\end{align*}
Moreover, the above equality holds if and only if $\displaystyle{\frac{d_i+d_j}{d^2_i+d^2_j}}=\displaystyle{\frac{d_k+d_{\ell}}{d^2_k+d^2_{\ell}}}$ for any two edges $v_iv_j,\,v_kv_{\ell}\in E(G)$.
Using the above result, from the definition, we obtain
\begin{align*}
ISDD(G)&=\sum\limits_{v_iv_j\in E(G)}\,\frac{d_i\,d_j}{d^2_i+d^2_j}\\[3mm]
  &=\frac{1}{2}\,\sum\limits_{v_iv_j\in E(G)}\,\left[\frac{2d_i\,d_j}{d^2_i+d^2_j}+1-1\right]\\[3mm]
  &=\frac{1}{2}\,\sum\limits_{v_iv_j\in E(G)}\,\frac{(d_i+d_j)^2}{d^2_i+d^2_j}-\frac{m}{2}\\[3mm]
  &\geq \frac{M_1(G)^2}{2\,F(G)}-\frac{m}{2}.
\end{align*}
The first part of the proof is done.

\vspace*{3mm}

Moreover, the equality holds if and only if $\displaystyle{\frac{d_i+d_j}{d^2_i+d^2_j}}=\displaystyle{\frac{d_k+d_{\ell}}{d^2_k+d^2_{\ell}}}$ for any two edges $v_iv_j,\,v_kv_{\ell}\in E(G)$, that is, if and only if $G$ is a regular graph or $G$ is a semiregular bipartite graph or $G\in \Gamma_3$, by Lemma \ref{1q1}.
\end{proof}

\vspace*{4mm}

\section{Concluding Remarks}

The inverse symmetric division deg index has been investigate in this paper in view of its mathematical properties.
In this paper we give some upper and lower bounds on $ISDD$ index of graphs in terms of several graph parameters, and characterize the corresponding extremal graphs. Moreover, we have established two relations, first one involving three topological indices $ISDD$, $GA$ and $M_2$, and the second one related with $ISDD$, $M_1$ and $F$.
Future studies on inverse symmetric division deg index $(ISDD)$ might focus on the maximal and minimal graphs for the class of $p\,(\geq 1)$-cyclic graphs.

\vspace*{6mm}
%
%
%
%
%
%
%
%


\begin{thebibliography}{99}

\bibitem{A1} H. Ahmed, A. Saleh, R. Ismail, R. Salestina M, A. Alameri, Computational analysis for eccentric neighborhood Zagreb indices and their significance, {\it Heliyon} {\bf 9} (2023) e17998.

\bibitem{AA1} A. M. Albalahi, A. Ali, On the inverse symmetric division deg index of unicyclic graphs, {\it Computation} {\bf 10} (2022) 181.

\bibitem{AEM} A. Ali, S. Elumalai, T. Mansour, On the symmetric division deg index of molecular graphs, {\it MATCH Commun. Math. Comput. Chem.} {\bf 83} (2020) 205--220.


\bibitem{BM}  J. A. Bondy, U. S. R. Murty, {\it Graph Theory}, Springer, 2008.

\bibitem{BHD1} L. Buyantogtokh, B. Horoldagva, K. C. Das, On general reduced second Zagreb index of graphs, {\it Mathematics} {\bf 10} (2022) 3553.

\bibitem{CA} D. de Caen, An upper bound on the sum of squares of degrees in a graph, Discrete Math. \textbf{185} (1998) 245--248.

\bibitem{E13}
B. Chaluvaraju, H. S. Boregowda, I. N. Cangul,  Some Inequalities for the First General Zagreb Index of Graphs and Line Graphs, {\it Proc. Natl. Acad. Sci., India, Sect. A Phys. Sci.} \textbf{91} (2021) 79–88.
 
\bibitem{CC} Z. Che, Z. Chen, Lower and upper bounds of the forgotten topological index, {\it MATCH Commun Math Comput Chem} {\bf 76} (2016) 635--648.

\bibitem{CD1} X. Chen, K. C. Das, Solution to a conjecture on the maximum $ABC$ index of graphs with given chromatic number, {\it Discrete Appl. Math.} {\bf 251} (2018) 126--134.

 \bibitem{CW1} M. Cheng, L. Wang, A lower bound for the Harmonic index of a graph with minimum degree at least three, {\it Filomat} {\bf 30} (8) (2016) 2249--2260.

\bibitem{DA1} K. C. Das, On geometric-arithmetic index of graphs, {\it MATCH Commun. Math. Comput. Chem.} {\bf 64} (2010) 619--630.

\bibitem{DA0} K. C. Das, I. Gutman, B. Furtula,  On first geometric-arithmetic index of graphs, {\it Discrete Appl. Math.} {\bf 159} (2011) 2030--2037.

\bibitem{DA3} K. C. Das, S. Mondal, On Neighborhood inverse sum indeg index of molecular graphs with chemical significance, {\it Information Sciences}  {\bf 623} (2023) 112--131.

\bibitem{DA4} K. C. Das, M. Mateji\'c, E. Milovanovi\'c, I. Milovanovi\'c, Bounds for symmetric division deg index of graphs, {\it Filomat} {\bf 33} (2019) 683--698.

\bibitem{DA5} K. C. Das, Y. Shang, Some extremal graphs with respect to Sombor index, {\it Mathematics} {\bf 9} (2021) 1202.

\bibitem{DZT} Z. Du, B. Zhou, N. Trinajsti\'c, On geometric-arithemetic indices of (molecular) trees, unicyclic graphs and bicyclic graphs, {\it MATCH Commun. Math. Comput. Chem.} {\bf 66} (2011) 681--697.

\bibitem{EMR} S. Elumalai, T. Mansour, M. A. Rostami, On the bounds of the forgotten topological index, {\it Turkish Jour. Math.} {\bf 41} (6) (2017) 1687--1702.

\bibitem{E} E. Estrada, Characterization of 3D molecular structure, {\it Chem. Phys. Lett.} {\bf 319} (2000) 713--718.

\bibitem{FDG} B. Furtula, K. C. Das, I. Gutman, Comparative analysis of symmetric division deg index as potentially useful molecular descriptor, {\it Int. J. Quantum Chem.} {\bf 118} (2018) e25659.

\bibitem{FGV} B. Furtula, A. Graovac, D. Vuki\v{c}evi\'c, Atom-bond connectivity index of trees, {\it Discrete Appl. Math.} {\bf 157} (2009) 2828--2835.

\bibitem{FG} B. Furtula, I. Gutman, A forgotten topological index, {\it J. Math. Chem.} {\bf 53} (2015) 1184--1190.

\bibitem{M1} M. Ghorbani, S. Zangi, N. Amraei, New results on symmetric division deg index, {\it Jour. Appl. Math. Computing} {\bf 65} (2021) 161--176.

\bibitem{GT}  I. Gutman, N. Trinajsti\'c, Graph theory and molecular orbitals, total $\pi$-electron energy of alternant hydrocarbons, {\it Chem. Phys. Lett.} {\bf 17} (1972) 535--538.

\bibitem{HO1} H. Hosoya, Topological index. A newly proposed quantity characterizing the topological nature of structural isomers of saturated hydrocarbons, {\it Bull. Chem. Soc. Japan} {\bf 44} (1971) 2332--2339.

\bibitem{E11}
S. R. Islam, M. Pal,  Multiplicative Version of First Zagreb Index in Fuzzy Graph and its Application in Crime Analysis, {\it Proc. Natl. Acad. Sci., India, Sect. A Phys. Sci.} \textbf{94} (2024) 127–141.


\bibitem{LM1} J. B. Liu, M. M. Mateji\'c, E. I. Milovanovi\'c, I. \v{Z}. Milovanovi\'c, Some new inequalities for the forgotten topological index and coindex of graphs, {\it MATCH Commun. Math. Comput. Chem.} {\bf 84} (2020) 719--738.
    
    
\bibitem{M1}E. Milovanovi\'c,  I. Milovanovi\'c, M. Jamil, Some properties of the Zagreb indices. {\it Filomat} {\bf 32} (2018) 2667--2675.

\bibitem{E12}
C. Natarajan, S. K. Ayyaswamy, D. Sarala et al. On F-index of Certain Generalized Thorny Graphs, {\it Proc. Natl. Acad. Sci., India, Sect. A Phys. Sci.} \textbf{91} (2021) 269–272.

\bibitem{RA} J. Radon, \"Uber die absolut additiven Mengenfunktionen, Wiener-Sitzungsber., (IIa), {\bf 122} (1913) 1295--1438.

\bibitem{R}  M. Randi\'c, On characterization of molecular branching, {\it J. Am. Chem. Soc.} {\bf 97} (1975) 6609--6615.

\bibitem{RSD1} Z. Raza, L. Saha, K. C. Das, On inverse symmetric division deg index of graphs, {\it RAIRO - Operations Research} {\bf 57} (2023) 3223--3236.

 \bibitem{20} V. S. Shegehalli, R. Kanabur, Arithmetic--geometric indices of path graph, {\it J. Comput. Math. Sci.} {\bf 16} (1) (2015) 19--24.

 \bibitem{21} G. H. Shirdel, H. Rezapour, A. M. Sayadi, The hyper--Zagreb index of graph operations, {\it Iran. J. Math. Chem.} {\bf 4} (2013) 213--220.

 \bibitem{22} R. Todeschini,  V. Consonni, New local vertex invariants and molecular descriptors based on functions of the vertex degrees, {\it MATCH Commun. Math. Comput. Chem.} {\bf 64} (2) (2010) 359--372.

 \bibitem{23} A. Vasilyev, Upper and lower bounds of symmetric division deg index, {\it Iran. J. Math. Chem.} {\bf 5} (2) (2014) 91--98.

 \bibitem{25} D. Vuki\v cevi\'c, Bond additive modeling 2. Mathematical properties of max-min rodeg index,  {\it Croat. Chem. Acta} {\bf 83} (2010) 261--273.

 \bibitem{26} D. Vuki\v cevi\'c, B. Furtula, Topological index based on the ratios of geometrical and arithmetical means of end--vertex degrees of edges, {\it J. Math. Chem.} {\bf 46} (2009) 1369--1376.

\bibitem{VG}  D. Vuki\v{c}evi\'c, M. Gasperov, Bond additive modeling 1. Adriatic indices, {\it Croat. Chem. Acta} {\bf 83} (2010) 243--260.

\bibitem{WY} H. Wang, S. Yuan, On the sum of squares of degrees and products of adjacent degrees, {\it Discrete Math.} {\bf 339}\,(2016) 1212--1220.

\bibitem{WI1} H. Wiener, Structural determination of paraffin boiling points, {\it J. Am. Chem. Soc.} {\bf 69} (1947) 17--20.

\bibitem{xu2011} K. Xu, The Zagreb indices of graphs with a given clique number, {\it Appl. Math. Lett.} \textbf{24} (2011) 1026--1030.

\bibitem{yua10} Y. Yuan, B. Zhou, N. Trinajsti\'c, On geometric-arithmetic index. {\it J. Math. Chem.} {\bf 47} (2010) 833--841.

\end{thebibliography}
\end{document}